\documentclass[reqno]{amsart}

\usepackage{amsmath,amssymb,amsthm,amsfonts,
	enumerate,hyperref,cleveref}

\usepackage{tikz-cd}

\usepackage{scalerel}

\hypersetup{
	colorlinks=true,
	linkcolor=blue,
	filecolor=magenta,      
	urlcolor=cyan,
	pdftitle={Overleaf Example},
	pdfpagemode=FullScreen,
}
\usepackage{dsfont}
\usepackage[all]{xy}
\usepackage[T1]{fontenc}
\usepackage{tikz}
\usepackage{pgfplots}
\pgfplotsset{compat=1.15}
\usepackage{mathrsfs}
\usetikzlibrary{arrows}
\newtheorem{lemma}{Lemma}[section]
\newtheorem{theorem}[lemma]{Theorem}
\newtheorem{proposition}[lemma]{Proposition}
\newtheorem{corollary}[lemma]{Corollary}

\crefrangeformat{equation}{#3(#1)#4--#5(#2)#6}

\crefname{thrm}{Theorem}{Theorems}
\crefname{theorem}{Theorem}{Theorems}
\crefname{lem}{Lemma}{Lemmas}
\crefname{cor}{Corollary}{Corollaries}
\crefname{prop}{Proposition}{Propositions}
\crefname{defn}{Definition}{Definitions}
\crefname{exm}{Example}{Examples}
\crefname{rem}{Remark}{Remarks}
\crefname{conj}{Conjecture}{Conjectures}
\crefname{quest}{Question}{Questions}
\crefname{section}{Section}{Sections}
\crefname{equation}{\unskip}{\unskip}
\crefname{enumi}{\unskip}{\unskip}
\crefname{subsection}{Subsection}{Subsections}

\begin{document}
	\title{Numerical index one characterises commutative JB$^*$-triples}	
	
	\author[L. Li]{Lei Li}
	\address[L. Li]{School of Mathematical Sciences and LPMC, Nankai University, 300071 Tianjin, China.}
	\email{leilee@nankai.edu.cn}
	
	\author[S. Liu]{Siyu Liu}
	\address[S. Liu]{School of Mathematical Sciences and LPMC, Nankai University, 300071 Tianjin, China.}
	\email{760659676@qq.com}
	
	\author[A.M. Peralta]{Antonio M. Peralta}
	\address[A.M. Peralta]{Instituto de Matem{\'a}ticas de la Universidad de Granada (IMAG), Departamento de An{\'a}lisis Matem{\'a}tico, Facultad de
		Ciencias, Universidad de Granada, 18071 Granada, Spain.}
	\email{aperalta@ugr.es}
	
	\subjclass[2010]{Primary 47A12 Secondary 46L05, 47A30, 46B20, 17C65}
	\keywords{numerical index one, JB$^*$-triple, C$^*$-algebra, commutativity, JB$^*$-algebra, associativity.} 
	
	\begin{abstract} We establish that a JB$^*$-triple $E$ is commutative if and only if its numerical index $n(E)$ is one. As a consequence, $n(E) = 1$ implies $n(E^*) = n(E^{**}) = 1$. This settles a question left open in the literature since 2008.
	\end{abstract}
	
	\maketitle
	
	
\section{Introduction} 

The numerical index of a Banach space stands as one of the most intensively studied geometric notions in functional analysis, bridging operator theory and the structural geometry of Banach spaces \cite{BonDunBookI,BonDunBookII}. Geometrically, this notion dictates the equivalence between the operator norm and the numerical radius, determining whether the latter can serve as an alternative, geometrically distinct norm on the space of bounded linear operators. Formally, let $X$ be a Banach space with topological dual $X^*$. For any operator $T$ in the space $L(X)$, of all bounded linear operators on $X$, its \emph{numerical radius} is given by
\[
v(T) = v_{B(X)}(T) = \sup \Big\{ |x^*(Tx)| : (x, x^*) \in X \times X^* : \|x\| = \|x^*\| = x^*(x) = 1 \Big\},
\]
and the (\emph{spatial}) \emph{numerical index} of $X$, $n(X),$ is the largest constant satisfying $n(X) \|T\| \leq v(T)$ for all $T \in L(X)$, explicitly defined by
\[
n(X) = \inf \Big\{ v(T) : T \in L(X), \|T\| = 1 \Big\}.
\]
Whenever $n(X) > 0$, the numerical radius $v(\cdot)$ constitutes an equivalent alternative norm on $L(X)$. It is known that the values of the numerical index exhaust the interval $[0,1]$ for real Banach spaces. In the complex framework, the sharp lower bound established by the Bohnenblust--Karlin theorem implies that the values cover the interval $[e^{-1}, 1]$ (cf.~\cite{DuncanMcGregorPryceWhite1970}).\smallskip

The numerical indices of $X$ and $X^*$ are intimately related, though they may differ in general. Indeed, at the operator level, it is known that $v_{B(X)}(T)=v_{B(X^*)}(T^*)$ for each $T\in B(X)$, and hence $n(X^*)\leq n(X)$ (see \cite[\S 9, Corollary 6]{BonDunBookI}). The long-standing open question of whether the inequality $n(X) \geq n(X^*)$ can be strict was definitively settled in \cite{BoykoKadMar2007numIndexDual}, where the authors constructed a Banach space with $n(X) > n(X^*)$ (see also \cite[Example 4.3]{MarJFA2008}). This pathology cannot occur in the context of C$^*$-algebras due to their rigid algebraic structure. Indeed, a classical result by Huruya~\cite{HuruyaPAMS1977} establishes that for any C$^*$-algebra $A$, $n(A) = 1$ if $A$ is commutative, and $n(A) = \frac{1}{2}$ otherwise. Since $A$ and its bidual, $A^{**},$ share the same algebraic status (commutative or non-commutative), Huruya's dichotomy implies $n(A^{**}) = n(A) $. By the standard duality bounds $n(A^{**}) \leq n(A^*) \leq n(A)$, it follows that $n(A) = n(A^*)$ for every C$^*$-algebra. A similar conclusion holds for JB$^*$-algebras since, by a result due to Kaidi, Morales, and A. Rodríguez-Palacios (see \cite{KadMorRodPal2001} or \cite[Proposition 3.5.44 and comments in \S 2.1.47 and page 422]{CabRodBookV1}), for each JB$^*$-algebra  $\mathcal{A}$ we have $n(\mathcal{A}) = 1$ whenever $\mathcal{A}$ is associative, and $n(\mathcal{A}) = 1/2$ otherwise. For the sake of brevity, technical definitions and some basic background, including references, are given in the next subsection. \smallskip

One of the most persistent open questions regarding the numerical index, which already appeared in the 2008 paper by Martín~\cite{MarMathNach2008} (see also~\cite{OikhbergMR2008}), asks whether $n(X)=1$ implies $n(X^*) = 1$ when $X$ is a JB$^*$-triple or the predual of a JBW$^*$-triple (i.e. a JB$^*$-triple which is also a dual Banach space). \smallskip

In the case that $X$ is a JBW$^*$-triple predual, Martín proved in \cite[Corollary 2.2]{martin09positive} that $v(X) = v(X^*)$, which solves the problem for JBW$^*$-triple preduals. Cabezas and the third author of this note gave a positive answer to the question above when $X$ is a JBW$^*$-triple (see \cite[Theorem~4]{CaPe24}). However, the problem for general JB$^*$-triples has remained open until the present note. We remark that a positive answer to our problem would follow if one could prove the stronger conclusion that every JB$^*$-triple with numerical index $1$ is commutative (cf. \cite[Lemma 3]{CaPe24}). \smallskip

The main conclusion of this note establishes that a JB$^*$-triple $E$ is commutative if and only if it has numerical index one, that is, $n(E)=1$ (see Corollary~\ref{c characterization of commutative JB*-triples as those with numerical index one}). Consequently, for any JB$^*$-triple $E$ we have $$n(E)=1 \Leftrightarrow n(E^{*})=1\Leftrightarrow  n(E^{**})=1.$$ We actually prove a stronger result showing that if a JB$^*$-triple $E$ can be embedded as a weak${}^{*}$-dense JB$^*$-subtriple of a JBW$^*$-triple of the form $\displaystyle W =\bigoplus_{j\in \Gamma}^{\ell_{\infty}} C_j$, where $\{C_j\}_{j\in \Gamma}$ is a family of Cartan factors satisfying one of the following statements: \begin{enumerate}[$(a)$]\item for some index $j_{0}$, the Cartan factor $C_{j_{0}}$ has $\text{rank} \geq 2$;\item every $C_{j}$ is a complex Hilbert space regarded as a type $1$ Cartan factor, and $\dim(C_{j_0}) \geq 2$ for some $j_0\in \Gamma$;  \end{enumerate} then the numerical index of the Banach space $E$ is less than or equal to $\frac{1}{2}$ (see Theorem~\ref{t non-commutative sufficient conditions for ni 1/2}). This result also provides a positive solution to \cite[Problem 1]{CaPe24}.\smallskip

Among the consequences of our main result, we show that a JB$^*$-triple $E$ has numerical index $1$ if and only if $E^*$ (or equivalently, $E^{**}$) has the alternative Daugavet property (see Corollary~\ref{c index 1 in preduals Daugavet}).\smallskip

The paper culminates with a new result proving that commutativity in JB$^*$-triples is characterized in terms of subtriples generated by pairs of elements; specifically, a JB$^*$-triple $E$ is commutative if and only if every JB$^*$-subtriple of $E$ generated by two elements is commutative. Moreover, commutativity is characterized by a Le Page-type inequality evaluated on pairs of elements in JB$^*$-subtriples generated by two elements (see Proposition~\ref{prop commutativity is charcterized by pairs}).

\subsection*{Background: JB$^*$-triple structures on function spaces}\ \ \smallskip

\noindent There are Banach spaces of continuous functions which are not closed for the usual (binary) pointwise product. To provide a classic example, let us recall that if $\mathbb{T}$ denotes the unit sphere of $\mathbb{C}$, a principal $\mathbb{T}$-bundle is a
subset $L$ of a locally convex Hausdorff complex linear space $X$ satisfying the following two properties:
\begin{enumerate}[$\bullet$]
\item $L$ is $\mathbb{T}$-symmetric, that is,  $\mathbb{T} L = L$.
\item  $0 \notin L$ and $L \cup \{0\}$ is compact.
\end{enumerate} When equipped with the supremum norm, the Banach space
$$ C_0^\mathbb{T}(L):=\{a\in C_0(L):a(\lambda t)=\lambda a(t)\text{ for every } (\lambda,t)\in\mathbb{T}\times L\}$$ is a closed subspace of the space $C_0(L)$ of all complex-valued continuous functions on $L\cup\{0\}$ vanishing at $0$. If $\Omega(\mathcal{C})$ stands for the spectrum of a commutative C$^*$-algebra $\mathcal{C}$ with dual space $\mathcal{C}^*$, the topological space $\mathbb{T}\ \Omega(\mathcal{C})\subseteq  \mathcal{C}^*$ admits a natural structure of principal $\mathbb{T}$-bundle where $\mu (\lambda \phi ) = (\mu \lambda) \phi$ ($\lambda,\mu \in \mathbb{T}$, $\phi\in \Omega(\mathcal{C})$). Given $f\in \mathcal{C} \cong C_0(\Omega(\mathcal{C}))$ we write $\tilde{f}$ for the function in $C_0^{\mathbb{T}} \left( \mathbb{T}\  \Omega(\mathcal{C}) \right)$ defined by $\tilde{f} (\lambda \phi ) = \lambda f(\phi )$. Clearly, the mapping $f\mapsto \tilde{f}$ is a surjective linear isometry from $\mathcal{C}\cong C_0\left( \Omega(\mathcal{C}) \right)$ onto $C_0^{\mathbb{T}}\left( \mathbb{T}\  \Omega(\mathcal{C}) \right)$ (cf. \cite[Proposition 10]{Ol74}). So, every commutative C$^*$-algebra is a $C_0^\mathbb{T}(L)$-space.  However, as shown in \cite[Corollary 1.13 and subsequent comments]{kaup83riemann} there are examples of $C_0^\mathbb{T}(L)$-spaces which are not isometrically isomorphic to any commutative C$^*$-algebra. Every $C_0^\mathbb{T}(L)$-space is a Lindenstrauss space, that is, its dual space is isometrically isomorphic to an $L_1(\mu)$ space for some measure $\mu$ (cf. \cite[Proposition 2.6]{FriRuCommutative} or \cite{Ol74}).\smallskip

If $f$ is a non-zero function in $C_0^\mathbb{T}(L)$, the functions $f^2$ and $f^*$ do not lie in $C_0^\mathbb{T}(L)$. However, triple products of the form $\{f,g,h\} = f g^* h$ remain in $C_0^\mathbb{T}(L)$ whenever $f,g,h\in C_0^\mathbb{T}(L)$. This triple product induces a structure of JB$^*$-triple on $C_0^\mathbb{T}(L)$.\smallskip 

Recall that a JB$^*$-triple is a complex Banach space $E$ with a
continuous ternary product $\{.,.,.\}$ symmetric and bilinear in
the outer variables and conjugate linear in the middle satisfying
$$L(x,y)  \{u,v,w\} = \{{L(x,y)u},v,w\} - \{ u,{L(y,x)v},w\} + \{u,v,{L(x,y)w}\},$$ such that $\|L(x,x)\| = \|x\|^2$ and $L(x,x)$ is an
hermitian operator on $E$ with non-negative spectrum, where
$L(x,y)$ is given by $L(x,y) z = \{x,y,z\}$ (see  \cite{kaup83riemann}).\medskip

The triple product on $C_0^\mathbb{T}(L)$ also satisfies the following identity
\begin{equation}\label{eq identity for commutative triples}
\{a,b,\{x,y,z\}\} = \{x,y,\{a,b,z\}\}, \hbox{ for all } x,y,z,a,b\in C_0^\mathbb{T}(L).	
\end{equation} A JB$^*$-triple satisfying the identity in \eqref{eq identity for commutative triples} is called \emph{commutative} or \emph{abelian}. Every commutative JB$^*$-triple is isometrically triple isomorphic to some $C_0^\mathbb{T}(L)$ for a principal $\mathbb{T}$-bundle $L$ (see \cite[Corollary 1.11]{kaup83riemann} or \cite[Theorem 4.2.7]{CabRodBookV1}). Non-commutative JB$^*$-triples include all non-commutative C$^*$-algebras, non-associative JB$^*$-algebras, spaces $B(H,K)$ of bounded linear operators between complex Hilbert spaces $H$ and $K$ with dim$(H)\geq 2$, and in particular, all complex Hilbert spaces with dimension $\geq 2$.\smallskip

As commented above, every commutative JB$^*$-triple is isomorphic to some $C_0^\mathbb{T}(L)$ for a principal $\mathbb{T}$-bundle $L$, and the class of function spaces defined by the latter objects is strictly wider than the set of all commutative C$^*$-algebras. Particular commutative JB$^*$-subtriples admit a finer representation. A JBW$^*$-triple is a JB$^*$-triple which is also a dual Banach space. Every JBW$^*$-triple admits a unique (isometric) predual and its triple product is separately weak$^*$-continuous \cite{BT1986}. The second dual of each JB$^*$-triple is a JBW$^*$-triple containing $E$ as a JB$^*$-subtriple \cite{dineen86complete}. Every commutative JBW$^*$-triple is a commutative von Neumann algebra (see \cite[Theorem 3]{CaPe24}).\smallskip 

Furthermore, the JB$^*$-subtriple $E_a$ generated by a single element $a$ in a JB$^*$-triple $E$ (i.e. the norm closure of the linear span of all odd powers of the form $a$, $a^{[3]}= \{a,a,a\},$ and $a^{[2n+1]} = \{a,a,a^{[2n-1]}\}$ for all $n\geq 2$), is isometrically JB$^*$-triple isomorphic to a commutative C$^*$-algebra of the form $C_0
(\Omega_a),$ for some (unique) locally compact Hausdorff space $\Omega_a\subseteq (0,\|a\|],$ such that $\Omega_a\cup \{0\}$ is compact. Here the symbol $C_0 (\Omega_a)$ denotes the Banach space of all complex-valued continuous functions on $\Omega_a\cup \{0\}$ vanishing at $0$ (cf. \cite[Corollary 1.15]{kaup83riemann} or \cite[Theorem 4.2.9]{CabRodBookV1}). The set $\Omega_a$ is known as the triple spectrum of $a$. Thus, given any function $f \in C_0(\Omega_a)$, there exists a unique element $f_t(a) \in E_a,$ and the continuous triple functional calculus $f \mapsto f_t(a)$ preserves linear combinations, triple products, and distances.\smallskip

Tripotents in JB$^*$-triples, that is, elements satisfying $\{e,e,e\}=e$, generalize partial isometries in C$^*$-algebras. Let $e$ be a tripotent in a JB$^*$-triple $E$. The eigenvalues of the operator $L(e,e)$ are contained in the set $\{0, \frac12, 1\}$ (see \cite[Fact 4.2.14]{CabRodBookV1}). The \emph{Peirce subspaces} associated with the tripotent $e$ are precisely the spaces 
\[
E_j(e) = \left\{ x\in E \Bigm| \{e,e,x\} = \frac j2 x \right\} \quad \text{for } j=0,1,2.
\] This spectral property leads to the Peirce decomposition $E = E_2(e) \oplus E_1(e) \oplus E_0(e)$. The corresponding canonical projections $P_j(e)$ ($j=0,1,2$) are contractive operators (see \cite[Fact 4.2.14]{CabRodBookV1}). The Peirce-$2$ subspace $E_2(e)$ is a JB$^*$-algebra with product $x\circ_e y=\{x,e,y\}$ and involution $x^{*_e}=\{e,x,e\}$ (cf. \cite[Theorem 2.2]{BraKaUp78}, \cite[Theorem 3.7]{kaup77jordan} or \cite[Corollary 4.2.30]{CabRodBookV1}).\smallskip

The \emph{range tripotent}, $r(a)$, of an element $a$ in a JBW$^*$-triple $W$ is the
smallest tripotent $e\in W$ satisfying that $a$ is positive
in the JBW$^*$-algebra $W_{2} (e)$. The existence of the range tripotent is guaranteed, for example, by \cite[Lemma
3.3]{edwards96compact}. The just quoted reference also shows that, by employing the continuous triple functional calculus, the range tripotent of $a$ can be computed as the limit in the weak$^*$ topology of $W$ of the sequence of the iterated cubic roots of $a$, that is,
\begin{equation}\label{equ range tripotent as weak* limit} r(a) = \text{w}^*\text{-}\lim_{n\to\infty} a^{[1/3^n]}.
\end{equation} In particular, $r(a)\neq 0$ whenever $a\neq 0$. A general JB$^*$-triple $E$ might not contain a single non-zero tripotent. However, for each $a\in E$ we can compute its range tripotent in $E^{**}$, which will be denoted by $r_{_{E^{**}}} (a)$ or simply by $r(a)$.  \smallskip

For convenience, we recall the definition of Cartan factors. These complex Banach spaces are classified into six types: type 1 Cartan factors include the spaces ${L}(H, K)$ of all bounded linear operators between complex Hilbert spaces $H$ and $K$ with triple product $\{a,b,c\}=\frac12 (a b^* c+ c b^* a)$; type 2 and type 3 Cartan factors are the subtriples of $L(H)$ given by all $t$-skew-symmetric ($A = -jA^*j$) and $t$-symmetric ($A = jA^*j$) operators, where $j: H \to H$ is a conjugation (defined as a conjugate-linear isometry of period 2); type 4 represents spin factors, which are essentially complex Hilbert spaces equipped with an equivalent norm; while type 5 ($M_{1,2}(\mathbb{O})$) and type 6 ($H_3(\mathbb{O})$) are the exceptional finite-dimensional spaces of $1 \times 2$ matrices and $3 \times 3$ hermitian matrices over the complex octonions $\mathbb{O}$ (see \cite{DanFri, horn87classification} and \cite[\S 7.1.4]{CabRodVol2} for complete details).\smallskip

Elements $a,b$ in a JB$^*$-triple $E$ are said to be \emph{orthogonal} (denoted by $a\perp b$) if $L(a,b)=0$. Several characterizations of orthogonality can be found in \cite[Lemma 1]{BurFerGarMarPe2008}. The \emph{rank} of $E$ is the smallest cardinal number $r$ such that the cardinality of any set of mutually orthogonal non-zero elements in $E$ does not exceed $r$.\smallskip

As observed before Remark 3.1 in \cite{LiLiuPe25} (see  \cite[\S 3]{BeLoPeRod2004} or \cite[Corollary in page 308]{DanFri}), a JB$^*$-triple has rank-one if and only if it is a complex Hilbert space regarded as a type 1 Cartan factor. Let $H \cong L(H,\mathbb{C})$ be a complex Hilbert space equipped with its structure of type 1 Cartan factor. The triple product is given by $\{x,y,z\} =\frac12 \langle x| y \rangle z + \frac12 \langle z| y \rangle x$, for all $x,y,z\in H$, where $\langle \cdot | \cdot \rangle$ denotes the inner product of $H$. Clearly, $H$ is a JBW$^*$-triple, and every norm-one  element in $H$ is a tripotent. The range tripotent of each non-zero element $x\in H$ is $r(x) = \frac{x}{\|x\|}$. Furthermore, the JB$^*$-subtriple of $H$ generated by $x$ satisfies $H_x = \mathbb{C} r(x)$, $\Omega_x=\{\|x\|\}$, and for each function $f:\{\|x\|\}\to \mathbb{C}$, we have $f_t (x) = f (\|x\|)\ r(x) = f (\|x\|) \frac{x}{\|x\|} $.\smallskip

Let $E$ be a JB$^*$-triple. A subset $S \subseteq E \setminus \{0\}$ is called an orthogonal set if $x \perp y$ for all distinct elements $x,y \in S$. The \emph{rank} of $E$, denoted by $\operatorname{rank}(E)$, is the smallest cardinal number $r$ such that the cardinality of any orthogonal set in $E$ does not exceed $r$.\smallskip

A celebrated result by W. Kaup in \cite[Proposition 5.5]{kaup83riemann} assures that a linear bijection between JB$^*$-triples is a linear isometry if and only if it is a triple isomorphism (i.e. it preserves triple products). Consequently, the triple product of a JB$^*$-triple is unique, that is, if a complex Banach space $E$ admits two triple products inducing a JB$^*$-triple structure on $E$, both products are identical.\label{page uniqueness of the triple product}\smallskip

A subspace $I$ of a JB$^*$-triple $E$ is called an \emph{inner ideal} (respectively, a \emph{triple ideal}) of $E$ if $\{I,E, I\}\subseteq I$ (respectively, $\{E,E, I\}+\{E,I, E\}\subseteq I$). All inner ideals and triple ideals in this note will be assumed to be closed. Geometrically speaking, the triple ideals in $E$ are precisely the $M$-ideals of $E$ as a Banach space \cite[Theorem 3.2]{BT1986}. In the case of JBW$^*$-triples, weak$^*$-closed triple ideals are all $M$-summands which are actually orthogonally complemented \cite{Horn1987predual}. The inner ideal $E(a)$ generated by a single element $a$ in $E$ enjoys additional algebraic properties. It is known that ${E}(a) = \overline{\{a, {E}, a\}}^{\|.\|} = \overline{Q(a) ({E})}^{\|.\|}$ in ${E}$, and satisfies the following properties:
\begin{enumerate}[$(1)$]  
	\item ${E}(a)$ contains the JB$^*$-subtriple of $E$ generated by $a$.
	\item ${E}(a)$ is a JB$^*$-subalgebra of ${E}^{**}_2(r_{{\scaleto{E^{**}\mathstrut}{4pt}}} (a))$, where $r_{{\scaleto{E^{**}\mathstrut}{4pt}}} (a)$ denotes the range tripotent of $a$ in $E^{**}$.
	\item The weak$^*$-closure of ${E}(a)$ in ${E}^{**}$, $\overline{{E}(a)}^{w^*}$, identifies naturally with ${E}(a)^{**}$ and coincides with ${E}_{2}^{**}(r(a)).$
\end{enumerate} The reader can find a proof of the above properties in \cite[Proposition 2.1]{BuCHuZa2000MathScand}.\smallskip

We are actually interested on a suitable version of the above result. Suppose now that $E$ is a JB$^*$-triple, $W$ is a JBW$^*$-triple, and $\kappa: E\hookrightarrow W$ is an isometric triple embedding with weak$^*$-dense image. Let us fix an element $a\in E$. The range tripotent of $\kappa(a)$ in $W$ ($r_{{\scaleto{W\mathstrut}{4pt}}} (\kappa(a))$ in short) is, in principle, unrelated to $r_{{\scaleto{E^{**}\mathstrut}{4pt}}} (a)$.\smallskip

By \cite[Proposition 6]{BarDaHor1988}, there exists a weak$^*$-closed triple ideal $M$ of $E^{**},$ and a triple isomorphism $\Phi: M \to W$ making commutative the following diagram:
$$ \begin{tikzcd}[row sep=large, column sep = large]  E \arrow[hook,swap]{d}{\iota_{_{E}}} \arrow[hook]{r}{\kappa} &  W \\
	E^{**}  \arrow[rightarrow]{r}{\pi_{_M}} &  M \arrow[u, "\Phi"]
\end{tikzcd}
$$	In the diagram above the hooked arrows denote the canonical embeddings, and $\pi_{_{M}}$ stands for the natural projection of $E^{**}$ onto $M$. It is clear that $E\cong \iota(E)\cong\kappa(E)$ and $E(a)\cong \iota(E) (\iota(a))\cong\kappa(E) (\kappa(a))$, where the symbol $\cong$ means (isometrically) triple isomorphic. By a little abuse of notation we write  $r_{{\scaleto{E^{**}\mathstrut}{4pt}}} (a)$ for  $r_{{\scaleto{E^{**}\mathstrut}{4pt}}} (\iota_{_E}(a))$. 
 The triple isomorphism $\Phi$ is weak$^*$-continuous (cf. \cite[Corollary 3.22]{Horn1987predual}), and the same is clearly true for the projection $\pi_{_M}$. Since all maps in the diagram are triple homomorphisms, we can easily deduce from \eqref{equ range tripotent as weak* limit} that $\Phi \left(\pi_{_M} (r_{{\scaleto{E^{**}\mathstrut}{4pt}}} (a) )\right) = r_{{\scaleto{W\mathstrut}{4pt}}} (\kappa(a))$, and 
$$
\Phi \pi_{_M} \{\iota_{_E}(x), r_{{\scaleto{E^{**}\mathstrut}{4pt}}} (a), \iota_{_E}(y) \} = \{\kappa(x), r_{{\scaleto{W\mathstrut}{4pt}}} (\kappa(a)), \kappa(y) \},$$ for all $x,y\in E$. If we take $\iota_{_E}(x), \iota_{_E}(y)\in \iota_{_E}(E)(a),$ it follows from the previously commented result in \cite[Proposition 2.1]{BuCHuZa2000MathScand} that \begin{equation}\label{eq preservation of Jordan products}
\begin{aligned}
		\kappa  (E) (\kappa (a)) &= \Phi \pi_{_M} \iota_{_E}  (E) (\iota_{_E} (a)) \ni  \Phi \pi_{_M} \left(\iota_{_E}(x)\circ_{{\scaleto{{r_{{\scaleto{E^{**}\mathstrut}{4pt}}} (a)}\mathstrut}{6pt}}} \iota_{_E}(y)\right) \\
		&= \Phi \pi_{_M} \{\iota_{_E}(x), r_{{\scaleto{E^{**}\mathstrut}{4pt}}} (a), \iota_{_E}(y) \} = \{\kappa(x), r_{{\scaleto{W\mathstrut}{4pt}}} (\kappa(a)), \kappa(y) \} \\
		&= \kappa(x)\circ_{{\scaleto{{r_{{\scaleto{W\mathstrut}{4pt}}} (\kappa(a))}\mathstrut}{6pt}}} \kappa(y), \hbox{ and }	
	\end{aligned}
\end{equation}
\begin{equation}\label{eq preservation of involution} \begin{aligned}
		\kappa  (E) &(\kappa (a))= \Phi \pi_{_M} \iota_{_E}  (E) (\iota_{_E} (a)) \ni \Phi \pi_{_M} \left(\left( \iota_{_E}(x) \right)^{*_{{\scaleto{{r_{{\scaleto{E^{**}\mathstrut}{3pt}}} (a)}\mathstrut}{5pt}}}} \right)\\
		&= \Phi \pi_{_M} \{r_{{\scaleto{E^{**}\mathstrut}{4pt}}} (a), \iota_{_E}(x), r_{{\scaleto{E^{**}\mathstrut}{4pt}}} (a) \} = \{r_{{\scaleto{W\mathstrut}{4pt}}} (\kappa(a)), \kappa(x), r_{{\scaleto{W\mathstrut}{4pt}}} (\kappa(a)) \} = \left( \kappa(x) \right)^{*_{{\scaleto{{r_{{\scaleto{W\mathstrut}{3pt}}} (\kappa(a))}\mathstrut}{5pt}}}},	
	\end{aligned}
\end{equation} which together assure that $E(a)\cong \kappa  (E) (\kappa (a))$ is a JB$^*$-subalgebra of the JBW$^*$-algebra $W_2 (r_{{\scaleto{W\mathstrut}{4pt}}} (\kappa(a)))$. Furthermore, as we observed above, $\Phi \pi_{_M} \left( \iota_{_E}  (E) (\iota_{_E} (a)) \right) = \kappa  (E) (\kappa (a))$, and the mapping  $$\Psi= \Phi \pi_{_M}|_{\iota_{_E}  (E) (\iota_{_E} (a))} : \iota_{_E}  (E) (\iota_{_E} (a)) \to \kappa  (E) (\kappa (a))$$ is a Jordan $^*$-homomorphism. Actually $\Psi$ is a Jordan $^*$-isomorphism. Namely, if $0= \Psi (\iota_{E} (x)) = \Phi \pi_{_M} \iota_{_E} (x) =\kappa (x),$ it follows that $x=0$. We gather the conclusions in the next result in which we simplified the obvious identifications.
 
\begin{proposition}\label{p BCZ for subtriples} Let $\kappa: E\hookrightarrow W$ be an isometric triple embedding of a JB$^*$-triple inside a JBW$^*$-triple. Suppose that $\kappa$ has weak$^*$-dense image. Then the following statements hold for each element $a$ in $E$:\begin{enumerate}[$(a)$]
		\item $\kappa(E) (\kappa(a))$ is a JB$^*$-subalgebra of $W_2(r_{{\scaleto{W\mathstrut}{4pt}}} (\kappa(a)))$, where $r_{{\scaleto{W\mathstrut}{4pt}}} (\kappa(a))$ denotes the range tripotent of $\kappa(a)$ in $W$. The element $\kappa (a)$ is positive in $\kappa(E) (\kappa(a))$. 
	\item $\kappa(E) (\kappa(a))$ contains the JB$^*$-subtriple of $W$ generated by $\kappa(a)$. 
	\item $\kappa(E) (\kappa(a))$ is Jordan $^*$-isomorphic to the JB$^*$-algebra $E(a)$, where the latter is regarded as a JB$^*$-subalgebra of  ${E}^{**}_2(r_{{\scaleto{E^{**}\mathstrut}{4pt}}} (a))$ and  $r_{{\scaleto{E^{**}\mathstrut}{4pt}}} (a)$ denotes the range tripotent of $a$ in $E^{**}$. 
	\end{enumerate}
In particular, $\kappa(E) (\kappa(a))$ is associative if and only if $E(a)$ is.   
\end{proposition}
	
\section{The numerical index of noncommutative JB$^*$-triples}

This section is entirely devoted to proving that the numerical index of each non-commutative JB$^*$-triple must be smaller than or equal to $\frac12$. We begin with a technical observation derived from the Sait{o}-Tomita-Lusin theorem for JB$^*$-triples in \cite{BFMP06saito-tomima-lusin}.

\begin{lemma}\label{l Lusin} Let $E$ be a weak$^*$-dense JB$^*$-subtriple of a JBW$^*$-triple $\displaystyle W$. Let $e_1, \ldots, e_m$ be a finite family of mutually orthogonal minimal tripotents in $W$. Then the following statements hold:
\begin{enumerate}[$(a)$]
\item There exists a norm-one element $a\in E$ such that $$a = e_1+\ldots+ e_m + P_0( e_1+\ldots+ e_m) (a).$$
\item For each $1\leq j\leq  m$ there exists a norm-one element $b\in E$ such that $$b = e_1+\ldots+ e_j+ P_0( e_1+\ldots+ e_m) (b).$$
\end{enumerate} 
\end{lemma}

\begin{proof} $(a)$ If $W= E^{**}$, the desired conclusion follows from \cite[Theorem 3.3$(b)$]{BFMP06saito-tomima-lusin} combined with \cite[Lemma 1.6]{FriRuss85Crelles}. In the general case, the arguments preceding \cite[Theorem 3.2]{LiLiuPe26} can be combined with the aforementioned conclusion to get the result. \smallskip
	
$(b)$ The case $j=m$ follows from $(a)$. Assuming that $j<m$, we can apply statement $(a)$ to obtain two norm-one elements $a,\tilde{a}\in E,$ such that $$a = e_1+\ldots + e_j+ e_{j+1}+\ldots + e_m + P_0( e_1+\ldots+ e_m) (a),$$ and $$\tilde{a} = e_1+\ldots+ e_j - e_{j+1}-\ldots - e_m + P_0( e_1+\ldots+ e_m) (\tilde{a}).$$ The element $b= \frac{a+ \tilde{a}}{2}\in E$ satisfies the desired property.
\end{proof}

The first result of this section is devoted to provide sufficient conditions to guarantee that a JB$^*$-triple which is weak$^*$-densely embedded inside an $\ell_{\infty}$-sum of a family of Cartan factors has numerical index smaller than or equal to $1/2$.

\begin{theorem}\label{t non-commutative sufficient conditions for ni 1/2} Let $\{C_j\}_{j\in \Gamma}$ be an arbitrary (non-trivial) family of Cartan factors. Suppose $E$ is a weak$^*$-dense JB$^*$-subtriple of $\displaystyle W =\bigoplus_{j\in \Gamma}^{\ell_{\infty}} C_j$. Suppose that one of the following statements holds:
\begin{enumerate}[$(a)$]
\item For some subindex $j_0$ the Cartan factor $C_{j_0}$ has rank $\geq 2$.
\item Every $C_j$ is a complex Hilbert space regarded as type 1  Cartan factor, and dim$(C_{j_0}) \geq 2$ for some $j_0\in \Gamma$.  
\end{enumerate}	Then, the numerical index of the Banach space $E$ is smaller than or equal to $\frac12$. 
\end{theorem}

\begin{proof} Under hypothesis $(a)$ we follow similar arguments to those employed in \cite[Theorem 5$(c)$]{CaPe24} with the help of Lemma~\ref{l Lusin}. We include a modified proof here for completeness reasons. Suppose that there exists $j_0\in \Gamma$ such that $C_{j_0}$ has rank $\geq 2$. It is known that under these assumptions $C_{j_0}$ contains a JB$^*$-subtriple $\mathcal{B}$ isometrically isomorphic to $M_2 (\mathbb{C})$ or to $S_2(\mathbb{C}) = \left\{ \left(\begin{array}{cc}
		\alpha & \beta \\
		\beta & \delta \\
	\end{array}\right) : \alpha, \beta, \delta\in \mathbb{C}\right\}$ (the $3$-dimensional spin factor), with the additional property that every minimal tripotent in $\mathcal{B}$ is minimal in $C_{j_0}$ (cf. \cite[Lemma 3.10 and the discussion prior to it]{FerPeAdv2018}, or the proof of $(xiv)\Rightarrow (iii)$ in \cite[Theorem 2]{CaPe24}). Therefore, the tripotents $e_{11} =  \left(\begin{array}{cc}
	1 & 0 \\
	0 & 0 \\
	\end{array}\right),$ $e_{22} =  \left(\begin{array}{cc}
	0 & 0 \\
	0 & 1 \\
	\end{array}\right),$ and $v =  \left(\begin{array}{cc}
	\frac12 & \frac{i}{2} \\
	\frac{i}{2} & -\frac{1}{2} \\
	\end{array}\right)$ are minimal in $\mathcal{B}$, $C_{j_0}$, and $W$. By Lemma~\ref{l Lusin} there exist norm-one elements $a,b\in E$ satisfying $b = e_{11} + e_{22} + b_0,$ and $a = v + a_0,$ where $b_0= P_0(e_{11}+e_{22}) (b)$ and $a_0 = P_0(e_{11}+e_{22}) (a)$. Note that $v$ and $\tilde{v} =   \left(\begin{array}{cc}
	-\frac12 & \frac{i}{2} \\
	\frac{i}{2} & \frac{1}{2} \\
	\end{array}\right)$ are two orthogonal tripotents in $\mathcal{B}$ with $v + \tilde{v} =   \left(\begin{array}{cc}
	0 & i \\
	i & 0 \\
	\end{array}\right),$ $e_{11}+e_{22}\in W_2 (v+\tilde{v})$, $v+\tilde{v}\in W_2 (e_{11}+e_{22})$, and hence $P_0(e_{11}+e_{22}) = P_0(v+\tilde{v})$. Set $$\begin{aligned}
	c = Q(b) (a) &=\{b,a,b\} = \{e_{11}+e_{22},v,e_{11}+e_{22}\} + \{b_0,a_0,b_0\} \\
	&= \left(\begin{array}{cc}
		\frac12 & -\frac{i}{2} \\
		-\frac{i}{2} & -\frac{1}{2} \\
	\end{array}\right) +Q(b_0) (a_0) = v^* +Q(b_0) (a_0).
	\end{aligned}$$ By construction, $c$ lies in the inner ideal, $E(b),$ of $E$ generated by $b,$ and $c_0 = P_0(e_{11}+e_{22}) (c) = Q(b_0) (a_0)\perp W_2(e_{11}+e_{22})\ni v, v^*=\left(\begin{array}{cc}
	\frac12 & -\frac{i}{2} \\
	-\frac{i}{2} & -\frac{1}{2} \\
	\end{array}\right)$. Clearly, the range tripotent of $b$ in $W$ ($r_{{\scaleto{W\mathstrut}{3.8pt}}}$ in short) satisfies $r_{{\scaleto{W\mathstrut}{3.8pt}}} = e_{11}+ e_{22} + r_{{\scaleto{W\mathstrut}{3.8pt}}}^{{\scaleto{0\mathstrut}{3.8pt}}}$, where $r_{{\scaleto{W\mathstrut}{3.8pt}}}^{{\scaleto{0\mathstrut}{3.8pt}}}$ stands for the range tripotent of $b_0$ in $W$. Observe that $c_0, c\in W_2(r_{{\scaleto{W\mathstrut}{3.8pt}}})$. Working on the JBW$^*$-algebra $W_2(r_{{\scaleto{W\mathstrut}{3.8pt}}})$, it is not hard to check, via matrix products and the uniqueness of the triple product in a JB$^*$-triple (see \cite[Proposition 5.5]{kaup83riemann} and the comments in page~\pageref{page uniqueness of the triple product}), that
	 $$\begin{aligned}
		(c\circ_{r_{{\scaleto{W\mathstrut}{3pt}}}} c^{*_{r_{{\scaleto{W\mathstrut}{3pt}}}}}) \circ_{r_{_{\scaleto{W\mathstrut}{3pt}}}} c &=  \{ \{c, c,  {r_{{\scaleto{W\mathstrut}{3pt}}}} \}, {r_{{\scaleto{W\mathstrut}{3pt}}}} ,c\} = \{ \{v^*, v^*,  {r_{{\scaleto{W\mathstrut}{3pt}}}} \}, {r_{{\scaleto{W\mathstrut}{3pt}}}},v^*\} +  \{ \{c_0, c_0,  {r_{{\scaleto{W\mathstrut}{3pt}}}} \}, {r_{{\scaleto{W\mathstrut}{3pt}}}} ,c_0\} \\
		& =  \left( \left(\begin{array}{cc}
			\frac12 & -\frac{i}{2} \\
			-\frac{i}{2} & -\frac{1}{2} \\
		\end{array}\right) \circ \left(\begin{array}{cc}
		\frac12 & -\frac{i}{2} \\
		-\frac{i}{2} & -\frac{1}{2} \\
		\end{array}\right)^* \right) \circ \left(\begin{array}{cc}
		\frac12 & -\frac{i}{2} \\
		-\frac{i}{2} & -\frac{1}{2} \\
		\end{array}\right) \\
		&+  (c_0\circ_{r_{{\scaleto{W\mathstrut}{3pt}}}} c_0^{*_{r_{{\scaleto{W\mathstrut}{3pt}}}}}) \circ_{r_{{\scaleto{W\mathstrut}{3pt}}}} c_0 \\
		&= \frac12 \left(\begin{array}{cc}
			\frac12 & -\frac{i}{2} \\
			-\frac{i}{2} & -\frac{1}{2} \\
		\end{array}\right) +  (c_0\circ_{r_{{\scaleto{W\mathstrut}{3pt}}}} c_0^{*_{r_{{\scaleto{W\mathstrut}{3pt}}}}}) \circ_{r_{{\scaleto{W\mathstrut}{3pt}}}} c_0 \\
		&= \frac12 v^* +  (c_0\circ_{r_{{\scaleto{W\mathstrut}{3pt}}}} c_0^{*_{r_{{\scaleto{W\mathstrut}{3pt}}}}}) \circ_{r_{{\scaleto{W\mathstrut}{3pt}}}} c_0,
	\end{aligned}, $$ where the last two summands are orthogonal. We can similarly derive that  
	$$\begin{aligned}
		(c\circ_{r_{{\scaleto{W\mathstrut}{3pt}}}} c) \circ_{r_{{\scaleto{W\mathstrut}{3pt}}}} c^{*_{r_{{\scaleto{W\mathstrut}{3pt}}}}} &= \{ \{c, {r_{{\scaleto{W\mathstrut}{3pt}}}}, c \}, {r_{{\scaleto{W\mathstrut}{3pt}}}} , \{{r_{{\scaleto{W\mathstrut}{3pt}}}} ,c , {r_{{\scaleto{W\mathstrut}{3pt}}}} \} \} \\
		&=  \{ \{v^*, {r_{{\scaleto{W\mathstrut}{3pt}}}}, v^* \}, {r_{{\scaleto{W\mathstrut}{3pt}}}} , \{{r_{{\scaleto{W\mathstrut}{3pt}}}} ,v^* , {r_{{\scaleto{W\mathstrut}{3pt}}}} \} \} + \{ \{c_0, {r_{{\scaleto{W\mathstrut}{3pt}}}}, c_0 \}, {r_{{\scaleto{W\mathstrut}{3pt}}}} , \{{r_{{\scaleto{W\mathstrut}{3pt}}}} ,c_0 , {r_{{\scaleto{W\mathstrut}{3pt}}}} \} \}\\
		 &= \left(\left(\begin{array}{cc}
			\frac12 & -\frac{i}{2} \\
			-\frac{i}{2} & -\frac{1}{2} \\
		\end{array}\right) \circ  \left(\begin{array}{cc}
		\frac12 & -\frac{i}{2} \\
		-\frac{i}{2} & -\frac{1}{2} \\
		\end{array}\right) \right) \circ \left(\begin{array}{cc}
		\frac12 & -\frac{i}{2} \\
		-\frac{i}{2} & -\frac{1}{2} \\
		\end{array}\right)^* \\
		&+  (c_0\circ_{r_{{\scaleto{W\mathstrut}{3pt}}}} c_0) \circ_{r_{{\scaleto{W\mathstrut}{3pt}}}} c_0^{*_{r_{{\scaleto{W\mathstrut}{3pt}}}}} \\
		&= 0 +  (c_0\circ_{r_{{\scaleto{W\mathstrut}{3pt}}}} c_0) \circ_{r_{{\scaleto{W\mathstrut}{3pt}}}} c_0^{*_{r_{{\scaleto{W\mathstrut}{3pt}}}}} .
	\end{aligned}$$ Therefore, $(c\circ_{r_{{\scaleto{W\mathstrut}{3pt}}}} c^{*_{r_{{\scaleto{W\mathstrut}{3pt}}}}}) \circ_{r_{{\scaleto{W\mathstrut}{3pt}}}} c \neq (c\circ_{r_{{\scaleto{W\mathstrut}{3pt}}}} c) \circ_{r_{{\scaleto{W\mathstrut}{3pt}}}} c^{*_{r(b)}}.$ Taking into account Proposition~\ref{p BCZ for subtriples}, we conclude that $E(b)$, regarded as a JB$^*$-subalgebra of $W_2(r_{{\scaleto{W\mathstrut}{3pt}}})$, is not associative. The same proposition proves that $E(b)$ regarded as a JB$^*$-subalgebra of $E^{**}_2 (r_{{\scaleto{E^{**}\mathstrut}{3pt}}} (b))$ is not associative either. Recall that a JB$^*$-algebra is associative if and only if it does not contain non-zero $2$-nilpotent elements \cite[Theorem 1]{IoLouRod1989commutativity}. Proposition 2 in \cite{CaPe24} finally assures that $n(E)\leq 1/2$. \smallskip

We assume next that all Cartan factors in the family $\{C_j\}_{j\in \Gamma}$ have rank-one, and hence they are all complex Hilbert spaces regarded as type 1 Cartan factors, and one of them, say $C_{j_0},$ has dimension $\geq 2$. We can therefore write elements in $W$ as families of the form $(x_j)_{j\in \Gamma}$, where each $x_j$ is an element in the Hilbert space $C_j$. Observe that triple products and continuous triple functional calculus on $W$ can be computed component-wise.\smallskip 

We follow next some arguments from \cite{LiLiuPe26}. Take two norm-one elements $\xi,\eta\in C_{j_0}$ with $\xi \perp \eta$. Since the elements $e = (e_j)_j$ and $v = (v_j)_j$ with $e_{j_0} = \xi$, $v_{j_0}= \eta$, and $v_j =e_{j} =0$ otherwise, are minimal tripotents in $W$, by \cite[Theorem 3.2]{LiLiuPe26} (see also \cite[Theorem 3.3]{BFMP06saito-tomima-lusin}) there exist norm-one elements $a,b\in E$ satisfying $b = e +  P_0 (e) (b),$  and  $a= v+ P_0 (v) (a)$ (that is, $a_{j_0} = \eta$ and $b_{j_0} = \xi$, and the remaining components are bounded in norm by $1$). Note that $W$ enjoys a natural structure of $\ell_{\infty}(\Gamma)$-module with respect to the pointwise multiplication and $\ell_{\infty}(\Gamma)$-valued inner product. By Proposition~\ref{p BCZ for subtriples}, the inner ideal $E(a)$ is a JB$^*$-subalgebra of $W_2 (r_{{\scaleto{W\mathstrut}{3.8pt}}}(a)).$ For each $0<\beta<1$ we write $a^{[\beta]}\in E$ the element obtained via continuous triple functional calculus of the function $f(t):= t^{\beta}$ ($t\in[0,1])$ at the element $a$. Note that $\|a^{[\beta]}\| =1$.  The same arguments given in the proof of \cite[Theorem 1.2]{LiLiuPe26} prove the following statements:\begin{enumerate}[$(i)$]
	\item $(c^{\beta}_j)_j = c^{\beta} := b - \{a^{[\beta]},b,a^{[\beta]}\}^{*_{\scaleto{r_{{\scaleto{W\mathstrut}{3.8pt}}}(a)\mathstrut}{3.8pt}}} \in E\backslash E(a),$ $\| c^{\beta}\| \leq 2$, and the $j_0$-component of $c^{\beta}$ is precisely $\xi$;
	\item $\displaystyle \left\| \langle a | c^{\beta} \rangle \right\|_{_{\ell_{\infty(\Gamma)}}} \leq \frac{2 \beta}{(1 + 2 \beta)^{\frac{1}{2 \beta} + 1}},$ for all $0<\beta <1$. 
\end{enumerate}

We need to refine the element $c^{\beta}$ to reduce its norm. Consider the continuous function $g:\mathbb{R}^+_0\to \mathbb{R},$ $g(t) = \min\{1,t\}$, and set via continuous triple functional calculus in $E$ or in $W$, $d^{\beta} := g_{t}(c^{\beta})\in E$ ($\beta \in (0,1)$). Since the continuous triple functional calculus on $W$ is computed component-wise, the $j$-component of $d^{\beta}$ is precisely $c^{\beta}_j$ if $\|c^{\beta}_j\| \leq 1,$ and $d^{\beta}_j = \frac{c^{\beta}_j}{\|c^{\beta}_j\|}$ otherwise. Therefore, $\|d^{\beta}\| = 1$ and $d^{\beta}_{j_0} = \xi$.\smallskip

It is not hard to check that the inequality  \begin{equation}\label{eq ineq inner product for dbeta} \left\| \langle a | d^{\beta} \rangle \right\|_{_{\ell_{\infty(\Gamma)}}} \leq \frac{2 \beta}{(1 + 2 \beta)^{\frac{1}{2 \beta} + 1}},
\end{equation} holds for all $0<\beta <1$. Namely, $\left\| \langle a | d^{\beta} \rangle \right\|_{_{\ell_{\infty(\Gamma)}}} = \sup_{j\in \Gamma} \left| \langle a_j | d^{\beta}_j \rangle \right|_{j}$. If $\|c_j^{\beta}\|\leq 1$ we have $d_j^{\beta} = c_j^{\beta}$, and thus $ \left| \langle a_j | d^{\beta}_j \rangle \right|_{j}  = \left| \langle a_j | c^{\beta}_j \rangle \right|_{j} \leq \frac{2 \beta}{(1 + 2 \beta)^{\frac{1}{2 \beta} + 1}}$ in this case. If $\|c_j^{\beta}\|>1$, we have $ \left| \langle a_j | d^{\beta}_j \rangle \right|_{j}  = \left| \langle a_j | \frac{c^{\beta}_j}{\|c^{\beta}_j\|} \rangle \right|_{j} \leq  \left| \langle a_j | c^{\beta}_j \rangle \right|_{j} \leq \frac{2 \beta}{(1 + 2 \beta)^{\frac{1}{2 \beta} + 1}}$. This concludes the proof of  \eqref{eq ineq inner product for dbeta}.\smallskip

Since $d^{\beta},a\in W$, the operator $L(d^{\beta},a): W\to W$ belongs to $B(W)$. Having in mind that $d^{\beta},a\in E$ and $E$ is a JB$^*$-subtriple of $W$, we conclude that $L(d^{\beta},a) (E)\subseteq E$. Lemma 4 in \cite{CaPe24} assures that \begin{equation}\label{eq inequlities numerical radious in E and W} v_{_{B(E)}} (L(d^{\beta},a)|_{E}) \leq  v_{_{B(W)}} (L(d^{\beta},a)).
\end{equation} 

By construction $$\begin{aligned} 1/2 = 1/2 \|\xi\|&= \| \{\xi, \eta, \eta\}\| = \left\| \{d^{\beta}_{j_0},a_{j_0},a_{j_0}\} \right\|\leq \left\| \{d^{\beta},a,a\} \right\|\leq  \left\| L(d^{\beta},a) (a) \right\| \\ &\leq \left\| L(d^{\beta},a)|_{E}  \right\|_{_{B(E)}} \leq \left\| L(d^{\beta},a)  \right\|_{_{B(W)}} \leq  \|d^{\beta}\| \ \| a\| =1.
\end{aligned}	$$

The rest of the proof is devoted to finding a good estimation of $v_{_{B(W)}} (L(d^{\beta},a))$ to deduce that $n(E)\leq 1/2$. Observe that $L(d^{\beta},a)$ behaves as a diagonal operator on $W = \bigoplus_{j\in \Gamma}^{\ell_{\infty}} C_{j}$, that is, $$\begin{aligned}
L(d^{\beta},a) \big((h_j)_{j\in \Gamma}\big) &= \big( L(d_j^{\beta},a_j) (h_j)_{j\in \Gamma}\big) \\
&= \Big(\{d^{\beta}_j, a_j, h_j\} \Big)_{j\in \Gamma} = \Big(\frac12 \langle d^{\beta}_j| a_j\rangle h_j + \frac12 \langle h_j| a_j\rangle d^{\beta}_j \Big)_{j\in \Gamma} \\
&= \Big(\Big(\frac12 \langle d^{\beta}_j| a_j\rangle Id_{_{C_{j}}} + \frac12 d^{\beta}_j\otimes a_j \Big)  (h_j)  \Big)_{j\in \Gamma} \\
&= \Big( \frac12 \langle d^{\beta}_j| a_j\rangle Id_{_{C_{j}}} + \frac12 d^{\beta}_j\otimes a_j \Big)_{j\in \Gamma} \big((h_j)_{j\in \Gamma}\big).
\end{aligned}  $$ We write $L(d^{\beta}_j,a_j)$ for the $j$th-component of $L(d^{\beta},a)$. Each component is a bounded linear operator on the corresponding complex Hilbert space $C_j$ (being the transposed of $\frac12 \overline{\langle d^{\beta}_j| a_j\rangle} Id_{_{C_{j}}} + \frac12 a_j\otimes d^{\beta}_j $). Therefore, $L(d^{\beta},a)$ can be seen as the transposed of the (diagonal) bounded linear operator $$L(d^{\beta},a)_*= \Big( \frac12 \overline{\langle d^{\beta}_j| a_j\rangle} Id_{_{C_{j}}} + \frac12 a_j\otimes d^{\beta}_j \Big)_{j\in \Gamma}$$ defined on $\displaystyle W_* = \bigoplus_{j\in \Gamma}^{\ell_{1}} C_{j}^*\cong \bigoplus_{j\in \Gamma}^{\ell_{1}} C_{j}$. It is known from \cite[Corollary 9.6]{BonDunBookI} that $$v_{_{B(W)}}\big(L(d^{\beta},a)\big) = v_{_{B(W_*)}}\big( L(d^{\beta},a)_*\big),$$ and it is not hard to check that \begin{equation}\label{eq numerical radius as a supremum} \begin{aligned}
	v_{_{B(W)}}\big(L(d^{\beta},a)\big) &= v_{_{B(W_*)}}\big( L(d^{\beta},a)_*\big) = \sup_{j\in \Gamma} v_{_{B(C_j)}}\big( L(d_j^{\beta},a_j)_*\big) \\
	&=   \sup_{j\in \Gamma} v_{_{B(C_{j})}}\big( L(d_j^{\beta},a_j)\big).
\end{aligned}
\end{equation} 

We compute each summand separately. For $j = j_0$ we have $$L(d_{j_0}^{\beta},a_{j_0}) = \frac12 \langle \xi| \eta\rangle Id_{_{C_{j_0}}} + \frac12 \xi\otimes \eta = \frac12 \xi\otimes \eta.$$ It is well known that \begin{equation}\label{eq numerical radius for j0} v_{_{B(C_{j_0})}} \left(L(d_{j_0}^{\beta},a_{j_0})\right) = v_{_{B(C_{j_0})}} \left(\frac12 \xi\otimes \eta\right) = \frac14.
\end{equation}

Suppose now that $C_j = \mathbb{C}$ is one-dimensional. In this case, $L(d_{j}^{\beta},a_{j}): \mathbb{C}\to \mathbb{C}$ is the multiplication operator by the scalar $d_{j}^{\beta} \overline{a_{j}}\in \mathbb{C}$, and it is well known that \begin{equation}\label{eq numerical radius of L_j with one-dim} v_{_{B(C_{j})}} (L(d_{j}^{\beta},a_{j})) =\left\| L(d_{j}^{\beta},a_{j}) \right\| = |d_{j}^{\beta}| \ | \overline{a_{j}}| =  | \langle d_{j}^{\beta}  | {a_{j}} \rangle |  \leq \| \langle d^{\beta} | a \rangle \|_{_{\ell_{\infty}(\Gamma)}}.
\end{equation}

Suppose now that dim$(C_j)\geq 2$. We can clearly assume that $a_j\neq 0$ and we can find an orthonormal system $\left\{\frac{a_j}{\|a_j\|},\eta_j\right\}$ in the Hilbert space $C_j$ such that $$ d_{j}^{\beta} = \left\langle d_{j}^{\beta}  \Big| \frac{a_j}{\|a_j\|} \right\rangle \frac{a_j}{\|a_j\|} + \langle d_{j}^{\beta}  | \eta_j \rangle \eta_j,$$ where $|\langle d_{j}^{\beta}  | \eta_j \rangle| \leq \|d_{j}^{\beta}\|  \ \|\eta_j \| \leq 1.$\smallskip

Given two vectors $\xi,\eta$ in a Hilbert space $H$ we write $\xi\otimes \eta$ for the operator on $H$ given by $(\xi\otimes \eta) (h) = \langle h |\eta\rangle \xi$ ($h\in H$). It is known that, since $\xi\otimes \xi$ is a positive operator, we have $v_{_{B(H)}} (\xi\otimes \xi) = \|\xi\otimes \xi\| = \|\xi\|^2.$ Moreover, if $\xi\perp \eta$ the numerical radius of $\eta\otimes \xi$ satisfies $v_{_{B(H)}} (\eta\otimes \xi) = \frac12 \| \eta\otimes \xi\|.$  According to this notation and the triple product defined on a Hilbert space regarded as a type 1 Cartan factor, the operator $L(d_{j}^{\beta},a_{j}):{C}_{j}\to {C}_j$ satisfies 
$$\begin{aligned}
	L(d_{j}^{\beta},a_{j}) &= \frac12  \left\langle d_{j}^{\beta}  \Big| {a_j} \right\rangle Id_{_{C_{j}}} + \frac12 d_{j}^{\beta}\otimes a_{j} \\
	&= \frac12  \left\langle d_{j}^{\beta}  \Big| {a_j} \right\rangle Id_{_{C_{j}}} + \frac12 \left\langle d_{j}^{\beta}  \Big|  a_j \right\rangle \frac{a_j}{\|a_j\|} \otimes \frac{a_j}{\|a_j\|} + \frac12 \langle d_{j}^{\beta}  | \eta_j \rangle \eta_j \otimes a_{j}, 
\end{aligned} $$ and thus \begin{equation}\label{eq numerical sarious for Cj with dimension geq 2} \begin{aligned}
	v_{_{B(C_j)}} \left(	L(d_{j}^{\beta},a_{j})\right) &\leq \frac12  \left| \langle d_{j}^{\beta}  \Big| {a_j}\rangle \right| \    v_{_{B(C_j)}}\left(Id_{_{C_{j}}}  + \frac{a_j}{\|a_j\|} \otimes \frac{a_j}{\|a_j\|}\right) \\
	&+ \frac12 \left|\langle d_{j}^{\beta}  | \eta_j \rangle\right| \  v_{_{B(C_j)}}\left(\eta_j \otimes a_{j}\right) \\
	&=  \frac12  \left| \langle d_{j}^{\beta}  \Big| {a_j}\rangle \right| \    \left\|Id_{_{C_{j}}}  + \frac{a_j}{\|a_j\|} \otimes \frac{a_j}{\|a_j\|}\right\| \\
	&+ \frac12 \left|\langle d_{j}^{\beta}  | \eta_j \rangle\right| \frac12 \left\|\eta_j \otimes a_{j}\right\| \\
	&\leq  \left| \langle d_{j}^{\beta}  \Big| {a_j}\rangle \right| + \frac14 \left|\langle d_{j}^{\beta}  | \eta_j \rangle\right| \leq \| \langle d^{\beta} | a \rangle \|_{_{\ell_{\infty}(\Gamma)}} +\frac14,
\end{aligned}
\end{equation} for all $0<\beta <1$ and all $j\in \Gamma$ such that dim$(C_j)\geq 2$. It follows now from \eqref{eq numerical radius as a supremum}, \eqref{eq inequlities numerical radious in E and W}, \eqref{eq numerical radius for j0}, \eqref{eq numerical radius of L_j with one-dim}, and \eqref{eq numerical sarious for Cj with dimension geq 2} that \begin{equation}\label{eq inequ numerical radious of L} v_{_{B(E)}} \left(	L(d^{\beta},a)\right) \leq v_{_{B(W)}} \left(	L(d^{\beta},a)\right) \leq \| \langle d^{\beta} | a \rangle \|_{_{\ell_{\infty}(\Gamma)}} +\frac14 \ \ (\hbox{for all }0<\beta <1).
\end{equation} \smallskip

Given a positive $\varepsilon$, by combining \eqref{eq ineq inner product for dbeta} with the fact that $\displaystyle\lim_{\beta\to 0^+} \frac{2 \beta}{(1 + 2 \beta)^{\frac{1}{2 \beta} + 1}} =0$, we deduce the existence of $0<\beta_0 <1$ such that $ v_{_{B(E)}} \left(	L(d^{\beta_0},a)\right) < \frac14 + \varepsilon.$ Finally, the inequality $$\frac14 + \varepsilon > v_{_{B(E)}} \left(	L(d^{\beta_0},a)\right) \geq n(E) \ \left\| L(d^{\beta_0},a) \right\| \geq \frac12 n(E),$$ combined with the arbitrariness of $\varepsilon >0,$ proves that $n(E)\leq 1/2$ as desired. 
\end{proof}

As commented in the introduction, the problem of determining whether a JB$^*$-triple with numerical index one as a Banach space is necessarily commutative has remained open along the last decades. Assuming the stronger hypothesis that $M$ is a JBW$^*$-triple, it is established in \cite[Theorem 4]{CaPe24} that $n(M) =1$ if and only if $M$ is commutative, while $\frac{1}{e}\leq n(M)\leq \frac12$ otherwise. Our main result in this paper is the following theorem which proves that the same conclusion holds for general JB$^*$-triples. 

\begin{theorem}\label{t non-commutative JB*-triples have numercial index bounded by 1/2} The numerical index of the Banach space associated to every non-commutative JB$^*$-triple is smaller than or equal to  $\frac12$. 
\end{theorem}

Before entering into the details of the arguments leading to the proof of Theorem~\ref{t non-commutative JB*-triples have numercial index bounded by 1/2}, we briefly revisit the representation theory of JB$^*$-triples and a Gelfand-Naimark type theorem, which allows us to embed every JB$^*$-triple inside an $\ell_{\infty}$-sum of Cartan factors. Every JBW$^*$-triple $W$ can be decomposed as the direct sum of two orthogonal weak$^*$-closed ideals $W_{at}$ and $W_n$, called the atomic and non-atomic parts of $W$, respectively, such that  $W_{at}$ is the weak$^*$-closure of the linear span of all tripotents in $W$, $(W_{at})_*$ coincides with the norm closure of the linear span of all extreme points in the closed unit ball of $W_*$, and the unit ball of $\left(W_n\right)_*$ has no extreme points (see \cite[Theorems 1 and 2]{FriRuss85Crelles}). It is further known that $W_{at}$ is (isometrically) JB$^*$-triple isomorphic to an $\ell_{\infty}$-sum of a family of Cartan factors (cf. \cite[Proposition 2]{FriRuss86}).\smallskip

If $E$ is a general JB$^*$-triple, we can consider its second dual space, $E^{**}$, which is a JBW$^*$-triple (cf. \cite{dineen86complete} or \cite[Proposition 5.7.10]{CabRodVol2}). We write $\iota_{E}: E \hookrightarrow E^{**}$ for the canonical embedding. Let $\mathcal{A}$ denote the atomic part of $E^{**}$ (which is non-zero and an $\ell_{\infty}$-sum of a family of Cartan factors $\{C_j\}_{j\in \Gamma}$ by what we commented above). The Gelfand-Naimark theorem established by Y. Friedman and B. Russo in \cite[Propositions 1 and 2 and Theorem 1]{FriRuss86} asserts that if $\pi_{at}$ stands for the natural projection of $E^{**}$ onto $\mathcal{A}$, the mapping $\Phi_{_E} = \pi_{at}\circ  \iota_{E} : E\hookrightarrow  \mathcal{A}=\bigoplus_{j\in \Gamma}^{\ell_{\infty}} C_j$ is an isometric triple embedding with weak$^*$-dense image. We keep this notation along the paper. 

\begin{proof}[Proof of Theorem~\ref{t non-commutative JB*-triples have numercial index bounded by 1/2}] Suppose $E$ is a non-commutative JB$^*$-triple. Let us write  $\mathcal{A}=\bigoplus_{j\in \Gamma}^{\ell_{\infty}} C_j$ for the atomic part of $E^{**}$. As commented above, by the Gelfand-Naimark theorem for JB$^*$-triples, the JB$^*$-triple $E$ embeds isometrically as a weak$^*$-dense JB$^*$-subtriple of $\mathcal{A}$. If $C_j = \mathbb{C}$ for all $j\in \Gamma,$ the JB$^*$-triple $\mathcal{A}$ is commutative and hence $E$ must be commutative too, which contradicts our assumption. Therefore, the JBW$^*$-triple $\mathcal{A}$ satisfies one of the statements $(a)$ or $(b)$ in Theorem~\ref{t non-commutative sufficient conditions for ni 1/2}, and hence the just quoted result implies that $n(E)\leq 1/2$ as desired.
\end{proof}

The next corollary provides a complete positive solution to the problems posed in \cite[page 384]{MarMathNach2008}, \cite{OikhbergMR2008} and \cite[Problem 1 and comments prior to it]{CaPe24}.

\begin{corollary}\label{c characterization of commutative JB*-triples as those with numerical index one} A JB$^*$-triple $E$ is commutative if and only if it has numerical index one, that is, $n(E)=1$. Consequently, for any JB$^*$-triple $E$ we have $$n(E)=1 \Leftrightarrow n(E^{*})=1\Leftrightarrow  n(E^{**})=1.$$
\end{corollary}

\begin{proof} The ``only if'' implication follows from \cite[Lemma 3]{CaPe24}, while $n(E)=1$ implies that $E$ is commutative by Theorem~\ref{t non-commutative JB*-triples have numercial index bounded by 1/2}. The other equivalences can be found in \cite[Corollary 3 or Theorem 5]{CaPe24}.
\end{proof}

The equivalence of statements $(i)$ to $(vi)$ was established in \cite[Corollary 3]{CaPe24}. By virtue of Corollary~\ref{c characterization of commutative JB*-triples as those with numerical index one}, we can include the new equivalent condition $(vii)$, thus obtaining a complete generalization of \cite[Proposition 3.3]{MarMathNach2008} to JB$^*$-triples. Recall first that a Banach space $X$ satisfies the Daugavet property (respectively, the alternative Daugavet property) if $$\Vert{}\operatorname{Id}_X + T\Vert{} = 1 + \Vert{}T\Vert{} \quad \left(\text{respectively, } \max_{\vert{}\omega\vert{}=1} \Vert{}\operatorname{Id}_X + \omega T\Vert{} = 1 + \Vert{}T\Vert{}\right)$$ for every rank-one operator $T$ (cf. \cite{kssw,MarOikh2004,MarMathNach2008,KadMarBook2018,CaMarPe24}).

\begin{corollary}\label{c index 1 in preduals Daugavet} Let $E$ be a JB$^*$-triple. Then, the following are equivalent:
	\begin{enumerate}[$(i)$]
		\item $n(E^*) = 1$.
		\item $E^*$ has the alternative Daugavet property.
		\item $E^{**}$ has the alternative Daugavet property.
		\item $|\phi(u)| = 1$ whenever $u$ is an extreme point of the closed unit ball of $E^{**}$ and $\phi$ is an extreme point of the closed unit ball of $E^*$.
		\item The atomic part of $E^{**}$ is isometrically isomorphic to a commutative von Neumann algebra {\rm(}obviously atomic{\rm)}.
		\item $E$ {\rm(}equivalently, $E^{**}${\rm)} is commutative.
		\item $n(E) = 1$.
	\end{enumerate}
\end{corollary}

We conclude this paper with some additional facts about commutative JB$^*$-triples. In the first application, we derive a known result via a simple proof from our main theorem.

\begin{corollary}\label{c consequence of the charac of commutativity in terms of num index} Let $X$ be a complex Banach space. Then, the following statements hold: \begin{enumerate}[$(a)$]
\item Suppose $X$ admits a structure of C$^*$-algebra and a structure of JB$^*$-triple. Then the C$^*$-product on $X$ is commutative if and only if the triple product makes $X$ a commutative {JB}$^*$-triple.
\item Suppose $X$ admits a structure of JB$^*$-algebra and a structure of JB$^*$-triple. Then the Jordan product of $X$ is associative if and only if  the triple product makes $X$ a commutative {JB}$^*$-triple.
\end{enumerate}
\end{corollary}

\begin{proof} In both cases, since the triple product of a JB$^*$-triple is unique (cf. \cite[Proposition 5.5]{kaup83riemann}), the triple products in $A$ and $\mathfrak{A}$ are given by $\{a,b,c\} = \frac12 (a b^* c + c b^* a)$ and $\{a,b,c\} = (a\circ b^*)\circ c + (c\circ b^*)\circ a - (a\circ c)\circ b^*$, respectively, where juxtaposition denotes the C$^*$-product on $A$ and $\circ$ denotes the Jordan product on $\mathfrak{A}$. It is known that both statements can be now obtained via algebraic manipulations.\smallskip 
	
However, we can avoid the use of \cite[Proposition 5.5]{kaup83riemann} and the algebraic checking. Observing that the numerical index of $A$ and $\mathfrak{A}$ is determined solely by the underlying Banach space structure, Theorem~\ref{t non-commutative JB*-triples have numercial index bounded by 1/2}—combined with the fact that $A$ is commutative (respectively, $\mathfrak{A}$ is associative) if and only if $n(A) = 1$ \cite{HuruyaPAMS1977} (respectively, $n(\mathfrak{A}) = 1$ \cite{KadMorRodPal2001})—allows us to give a simple proof of the desired statements.	       
\end{proof}

Le Page's theorem affirms that a normed unital associative complex algebra $A$ is commutative if and only if there exists $\gamma > 0$ satisfying $\Vert{}ab\Vert{} \leq \gamma \ \! \Vert{}ba\Vert{}$ for all $a,b\in A$ (see \cite{LePage} or \cite[\S 2.1, Theorem 1]{Aupetit79}). It should be noted that the hypothesis on the existence of a unit element cannot, in general, be relaxed. If $A$ is a (not necessarily unital) C$^*$-algebra, a result due to Kaplansky shows that $A$ is commutative if and only if $A$ contains no non-zero element $c$ with $c^2 = 0$. Taking $a = c$ and $b = c^* c$, it follows that $ab \neq 0$ while $ba = 0$. Thus, the non-commutativity of $A$ can be simply reduced to the existence of elements $a, b \in A$ such that $ab \neq 0$ and $ba = 0$. Similarly, by \cite[Theorem 1]{IoLouRod1989commutativity} a (not necessarily unital) JB$^*$-algebra $\mathfrak{A}$ is non-associative if there exists an element $a\in \mathfrak{A}$ such that $(a\circ a^*)\circ a \neq 0$ and $(a\circ a)\circ a^*=0.$\smallskip

In the recent paper \cite{LiLiuPe26}, we establish that a JB$^*$-triple $E$ is commutative if and only if it satisfies a Le Page-type inequality; specifically, there exists $\gamma > 0$ such that $\Vert{}\{a,b,\{x,y,z\}\}\Vert{} \le \gamma \Vert{}\{x,y,\{a,b,z\}\}\Vert{}$ for all $a,b,x,y,z \in E$. That is, commutativity is equivalent to the existence of a universal constant $\gamma > 0$ for which this inequality holds across all quintuples of elements in $E$. We have not been able to find an $n$-tuple $(x_1,\ldots, x_n) \in E^n$ with $n \in \{3,4\}$ such that $\{x,y,\{a,b,z\}\} = 0$ while $\{a,b,\{x,y,z\}\} \neq 0$ for some choices of $a,b,x,y,z$ among these elements. Nevertheless, we can show that commutativity can be locally characterized in terms of JB$^*$-subtriples generated by pairs of elements.

\begin{proposition}\label{prop commutativity is charcterized by pairs} Let $E$ be a JB$^*$-triple. Then the following statements are equivalent:
\begin{enumerate}[$(a)$]
\item $E$ is commutative.
\item There exists a positive $\gamma >0$ satisfying $\|\{a,b,\{x,y,z\}\}\| \leq \gamma \| \{x,y,\{a,b,z\}\}\|$ for all $a,b,x,y,z\in E$.
\item There exists a positive $\gamma >0$ satisfying $\|\{a,c,\{c,a,a\}\}\| \leq \gamma \| \{c,a,\{a,c,a\}\}\|$ for all $a,c\in E$.
\item Every JB$^*$-subtriple of $E$ generated by two elements is commutative. 
\item For each JB$^*$-subtriple $F$ of $E$ generated by two elements, there exists a positive $\gamma(F) >0,$ depending on $F$, satisfying $\|\{a,b,\{x,y,z\}\}\| \leq \gamma(F) \| \{x,y,\{a,b,z\}\}\|$ for all $a,b,x,y,z\in F$.
\item For each JB$^*$-subtriple $F$ of $E$ generated by two elements, there exists a positive $\gamma(F) >0,$ depending on $F$, satisfying $\|\{a,c,\{c,a,a\}\}\| \leq \gamma(F) \| \{c,a,\{a,c,a\}\}\|$ for all $a,c\in F$.
\end{enumerate}
\end{proposition} 

\begin{proof} The equivalences $(a)\Leftrightarrow (b)\Leftrightarrow (c)$ and $(d)\Leftrightarrow (e)\Leftrightarrow (f)$ follow from \cite[Theorem~1.2 and its proof]{LiLiuPe26}. Clearly $(a)\Rightarrow (d)$, so we only need to prove that $(f)\Rightarrow (a)$. Proceeding by contradiction, assume that $E$ is not commutative. By \cite[Corollary 1]{CaPe24}, one of the following statements holds: \begin{enumerate}[$(1)$]
		\item There is an element $a$ in $E$ such that the inner ideal $E(a)$ contains a non-zero $2$-nilpotent element $b$ as JB$^*$-algebra, that is,  $\{b,r(a),b\}=0$.
		\item Every single generated inner ideal of $E$ is an associative JB$^*$-algebra and the atomic part of $E^{**}$ reduces to a $\ell_{\infty}$-sum of Hilbert spaces and at least one of them has dimension greater than or equal to $2$. 
	\end{enumerate}
	
In the first case the element $a$ is clearly non-zero. Let $F$ denote the JB$^*$-subtriple of $E$ generated by $a$ and $b$. By statement $(4)$ in the proof of \cite[Proposition 2.3]{LiLiuPe26}, the sequence $(a^{[\frac{1}{2n +1}]})_{n}$ is an approximate unit for the JB$^*$-algebra $(E(a), \circ_{_{r(a)}}, *_{_{r(a)}})$. Since the triple product on $E(a)$ is unique (cf. \cite[Proposition 5.2]{kaup83riemann}), we deduce that $$\{b,a^{[\frac{1}{2n +1}]},b\} = 2(b \circ_{_{r(a)}} a^{[\frac{1}{2n +1}]}) \circ_{_{r(a)}} b - (b\circ_{_{r(a)}} b) \circ_{_{r(a)}} a^{[\frac{1}{2n +1}]}$$ converges to  $(b\circ_{_{r(a)}} b) = \{b,r(a), b\}=0$ in norm. Consequently, \begin{equation}\label{eq first triple product tends to zero} \left\| \{a^{[\frac{1}{2n +1}]},b,\{b,a^{[\frac{1}{2n +1}]},b\}\} \right\|\to 0.
\end{equation} Similar arguments to those given above assure that the sequence $$\begin{aligned}
 \{a^{[\frac{1}{2n +1}]},b,b\} &= (a^{[\frac{1}{2n +1}]}\circ_{_{r(a)}} b^{*_{_{r(a)}}}) \circ_{_{r(a)}} b + (b\circ_{_{r(a)}} b^{*_{_{r(a)}}} )\circ_{_{r(a)}} a^{[\frac{1}{2n +1}]} \\
 &- (a^{[\frac{1}{2n +1}]}\circ_{_{r(a)}} b)\circ_{_{r(a)}} b^{*_{_{r(a)}}} 
\end{aligned}$$ converges to $(b\circ_{_{r(a)}} b^{*_{_{r(a)}}} )$ in norm, and consequently  
\begin{equation}\label{eq second triple product tends to non-zero} 
\{b, a^{[\frac{1}{2n +1}]},\{a^{[\frac{1}{2n +1}]},b,b\}\}\to b\circ_{_{r(a)}} (b\circ_{_{r(a)}} b^{*_{_{r(a)}}} ) =\frac12  U_{b} (b^{*_{_{r(a)}}} ) =\frac12 \{b,b,b\},
\end{equation} in norm. By combining \eqref{eq first triple product tends to zero}, \eqref{eq second triple product tends to non-zero}, $\| \{b,b,b\} \| = \|b\|^3 \neq 0$, and the fact that $b,a^{[\frac{1}{2n +1}]}\in F$ for all natural $n$, we get a contradiction with $(f)$.\smallskip

In the second case, the atomic part of $E^{**}$ decomposes into an $\ell_\infty$-direct sum of Hilbert spaces, where at least one space has dimension at least $2$. As observed in the proof of Theorem~\ref{t non-commutative sufficient conditions for ni 1/2}, a reapplication of the arguments in the proof of \cite[Theorem 3.2]{LiLiuPe26} guarantees the existence of two elements $a,b\in E$ with the following properties: \begin{enumerate}[$(i)$]\item $c^{\beta} := b - \{a^{[\beta]},b,a^{[\beta]}\}^{*_{r(a)}} \in E\backslash E(a)$ for all $0<\beta<1$, where $r(a)$ stands for the range tripotent of $a$ in $E^{**}$, and $a^{[\beta]}$ is determined by continuous triple functional calculus;
\item $\|\{c^{\beta},a,\{a,c^{\beta},a\}\}\| \leq \frac{4 \beta}{(1 + 2 \beta)^{\frac{1}{2 \beta} + 1}}\stackrel{\beta\to 0}{\longrightarrow} 0$, \item $\|\{a, c^{\beta},\{c^{\beta},a,a\}\}\|\geq \frac{1}{4}$ for all  $0<\beta<1$. 
\end{enumerate} Let $F$ denote the JB$^*$-subtriple of $E$ generated by $a$ and $b$. By observing that $F^{**}$ coincides with the weak$^*$-closure of $F$ in $E^{**}$, it can be easily deduced that $r(a)$ belongs to $F^{**}$, making it the range tripotent of $a$ in $F^{**}$. The properties of the continuous triple functional calculus show that $c^{\beta} \in F$ for all $0 < \beta < 1$. If we now combine $(ii)$ and $(iii)$, we obtain a contradiction with $(f)$. 
\end{proof}

\medskip
\medskip

\textbf{Acknowledgements} L. Li was supported by National Natural Science Foundation of China (grant No. 12571143). A.M. Peralta has been supported by MICIU/AEI/10.13039/501100011033 and ERDF/EU through the grants PID2021-122126NB-C31 and PID2025-167660NB-I00, by ``Maria de Maeztu'' Excellence Unit IMAG, reference CEX2020-001105-M, and by the Bureau of Foreign Experts Affairs, MOHRSS, PRC China (grant S20250924). \smallskip

This work was partly carried out during A.M. Peralta’s visit to Nankai University and the Chern Institute of Mathematics in June 2026, whose hospitality is gratefully acknowledged.\medskip

\subsection*{Statements and Declarations} 

All authors declare that they have no conflicts of interest to disclose.

\subsection*{Data availability}

There is no data associated for this submission.

	\end{document}